\documentclass[11pt,reqno]{amsart}

\usepackage{graphicx}
\usepackage{mathrsfs}
\usepackage{color}
\usepackage{comment}
\usepackage{amssymb}
\usepackage{esint}

\allowdisplaybreaks

\usepackage[margin = 1.2in] {geometry}

\usepackage[hyperindex,breaklinks]{hyperref}

\newtheorem{prop}{Proposition}%[section]
\newtheorem{theo}[prop]{Theorem}
\newtheorem*{theo*}{Theorem}
\newtheorem{lemm}[prop]{Lemma}
\newtheorem{coro}[prop]{Corollary}

\theoremstyle{definition}
\newtheorem{rema}[prop]{Remark}

\newtheorem{conj}[prop]{Conjecture}

\newcommand{\BB}{\mathbb{B}}
\newcommand{\CC}{\mathbb{C}}
\newcommand{\EE}{\mathbb{E}}
\newcommand{\RR}{\mathbb{R}}
\newcommand{\KK}{\mathbb{K}}
\newcommand{\HH}{\mathbb{H}}
\newcommand{\NN}{\mathbb{N}}
\newcommand{\OO}{\mathbb{O}}
\newcommand{\QQ}{\mathbb{Q}}
\renewcommand{\SS}{\mathbb{S}}
\newcommand{\ZZ}{\mathbb{Z}}

\newcommand{\bA}{\mathbf{A}}

\newcommand{\bH}{\mathbf{H}}

\newcommand{\bS}{\mathbf{S}}

\newcommand{\bW}{\mathbf{W}}
\newcommand{\bX}{\mathbf{X}}
\newcommand{\bY}{\mathbf{Y}}
\newcommand{\bZ}{\mathbf{Z}}

\newcommand{\bx}{\mathbf{x}}
\newcommand{\by}{\mathbf{y}}
\newcommand{\bz}{\mathbf{z}}

\newcommand{\cC}{\mathcal C}

\newcommand{\cQ}{\mathcal Q}

\newcommand{\cU}{\mathcal U}

\newcommand{\fc}{\mathfrak{c}}

\DeclareMathOperator{\tr}{tr}

\DeclareMathOperator{\Span}{span}

\DeclareMathOperator{\supp}{supp}

\DeclareMathOperator{\Ric}{Ric}
\DeclareMathOperator{\BiRic}{BiRic}
\DeclareMathOperator{\scal}{scal}
\DeclareMathOperator{\Rm}{Rm}

\usepackage{stmaryrd} 
\newcommand{\KN}{\mathbin{\owedge}}

\newcommand{\bangle}[1]{\left\langle #1 \right\rangle}

\newcommand{\eps}{\varepsilon}

\begin{document}

\title{Immersions with small normal curvature}
\author{Otis Chodosh}
\address{Department of Mathematics, Stanford University, Building 380, Stanford, CA 94305, USA}
\email{ochodosh@stanford.edu}
\author{Chao Li}
\address{Courant Institute, New York University, 251 Mercer St, New York, NY 10012, USA}
\email{chaoli@nyu.edu}

\begin{abstract}
We study Gromov's problem concerning minimal normal curvature immersions in the unit ball. In particular, we determine the minimal possible value of the normal curvature of an $S^n\times S^1$. We also prove a differentiable sphere theorem and an existence result for minimizers in this context. 
\end{abstract}

\maketitle

\section{Introduction}

Fix a smooth manifold $X$. For an immersion into the closed unit $N$-ball $f : X\to\BB^N$ with second fundamental form $\bA$, we define the \emph{maximal normal curvature} of $f$ as
\begin{equation}\label{eq:defn-nc}
\fc(f) : = \sup_{|u|=1} |\bA(u,u)|,
\end{equation}
where $|u|$ is measured with respect to the pullback metric $g=f^*g_{\RR^N}$. In this paper we study Gromov's \emph{minimal curvature invariant} (see \cite{Gromov:immersions1,Gromov:immersions2,Gromov:notes})
\begin{equation}\label{eq:grom-invar}
\cC_N(X) : = \inf \fc(f) 
\end{equation}
where the infimum is taken over immersions $f:X^n\to\BB^N$. It is straightforward to see that $\cC_N(X)$ is a decreasing function of $N$. Its behavior as $N$ tends to infinity captures intriguing topological properties of $X$, and is the main concern of this paper.

We list the closed $X$ for which the value of $\cC_N(X)$ is known for sufficiently large $N$. 
\begin{enumerate}
\item The $n$-sphere has $\cC_N(S^n) = 1$ for $N \geq n+1$ (achieved by the inclusion $\partial \BB^{n+1} \subset \partial \BB^{n+1}$). The maximum principle implies that $\cC_N(X) \geq 1$ always holds. 
\item The $n$-torus has $\cC_N(T^n) = \sqrt{\frac{3n}{n+2}}$ for $N \gg n$. Immersions achieving this bound were found by Gromov using spherical designs \cite{Gromov:immersions1,Gromov:immersions2}. These immersions were shown to be optimal by Petrunin \cite{Petrunin:tori}. 
\item The aforementioned result on $T^n$ implies that for any closed manifold $X$, $\cC_N(X) < \sqrt 3$ for $N\gg n$ \cite{Gromov:immersions1,Gromov:immersions2}. This may be obtained by composing an inflated Whitney embedding of $X$ into $\RR^M$ (with arbitrarily small normal curvature) and an immersion $T^{M}\to \BB^{N}$ with the optimal normal curvature.
\item The projective planes have: $\cC_N(\KK P^2) = \sqrt{\frac 43}$ for $\KK \in \{\RR,\CC,\HH,\OO\}$ and $N \geq 5$ for $\RR P^2$, $N\geq 8$ for $\CC P^2$, $N\geq 14$ for $\HH P^2$, $N \geq 26$ for $\OO P^2$. The ``Veronese immersions'' achieve these bounds (cf.\ Section \ref{subsec:Veronese}) and these immersions were shown to be optimal by Petrunin \cite{Petrunin:veronese}.  
\end{enumerate}

\subsection{The minimal curvature invariant of \texorpdfstring{$S^n\times S^1$}{Sn x S1}} One of our main results is to compute the minimal curvature invariant of $S^n \times S^1$ (for $N$ sufficiently large):
\begin{theo}\label{theo:SnS1}
For $n\geq 2$ we have $\cC_{N}(S^n \times S^1) = \sqrt{\frac 32}$ for $N \geq 2n+4$. 
\end{theo}
An immersion of $S^n\times S^1$ into $\CC P^{n+1}$ with low normal curvature was constructed by Naitoh \cite[Theorem 6.5]{Naitoh}. Lifting this to an immersion into $\SS^{2n+3}$ via the Hopf map and then composing with the inclusion into $\BB^{2n+4}$ yields an immersion with normal curvature $\sqrt{\frac32}$. In Section \ref{sec:examples} we describe a (seemingly) different method for finding a map with the same curvature bound. 

The bound in the other direction is obtained as follows. We recall that in \cite{Petrunin:tori}, Petrunin proved $\cC_N(X^n) \geq \sqrt{\frac{3n}{n+2}}$ for a closed manifold $X^n$ not admitting a metric of positive scalar curvature. A key idea in this paper is to generalize this method to other curvature conditions. 
\begin{prop}\label{prop:secPetrunin}
For $X$ closed, if $f : X \to \BB^N$ has $\fc(f) < \sqrt{\frac 32}$ then the induced metric $g=f^*g_{\RR^N}$ is conformally equivalent to a metric of positive sectional curvature. 
\end{prop}
Since $S^n\times S^1$ does not admit a metric of positive Ricci (much less sectional) curvature, this proves the lower bound. 

\begin{rema}\label{rema:RP3}
We observe that $\sqrt\frac32$ is also the normal curvature of the Veronese embedding of $\RR P^3$ (see Section \ref{subsec:Veronese} and note that $\varphi_{3,2}$ factors through $\RR P^3$). In light of Petrunin's work \cite{Petrunin:veronese}, it is natural to conjecture that if $f : X^3\to\BB^N$ has $\fc(f) \leq \sqrt{\frac 32}$ then $X^3$ is diffeomorphic to $S^3,\RR P^3$, or $S^2\times S^1$. 
\end{rema}

\subsection{Existence of minimizers}
The conformal factor in Proposition \ref{prop:secPetrunin} is explicit and uniformly bounded. As such, examining the proof of Proposition \ref{prop:secPetrunin}, we see that if $f:X^n\to \BB^N$ has $\fc(f) \leq \sqrt{\frac 32} - \delta$ then $(X^n,g=f^*g_{\RR^N})$ has intrinsic diameter $\leq D(\delta)$. In particular, this implies existence of (weak) minimizers in \eqref{eq:grom-invar}. More precisely the following result follows immediately from the intrinsic diameter bound and standard compactness results for immersions:
\begin{coro}
For $X^n$ closed, suppose that $\cC_N(X^n) < \sqrt{\frac32}$. Then there exists a $C^{1,1}$-immersion $f : X^n \to \BB^N$ so that $f$ has $|A(u,u)| \leq \cC_N(X)$ for almost every $p \in X$ and unit vector $u \in T_pX$. 
\end{coro}
One may define several notions of weak normal curvature of a map $f$ and we expect that $f$ have weak normal curvature \emph{equal} to $\cC_N(X)$ for most of these notions. We hope to investigate this, along with other properties of critical points elsewhere. 

\subsection{A differentiable sphere theorem for immersions with small normal curvatures}

The technique used to prove Proposition \ref{prop:secPetrunin} can be generalized to other curvature conditions. For example, by considering isotropic curvature, we obtain
\begin{theo}\label{theo:norm-curv-sphere-thm}
For $X^n$ closed, if $f : X \to \BB^N$ has $\fc(f) < \sqrt\frac 43$ then $X$ is diffeomorphic to a standard $n$-sphere. 
\end{theo}  
This answers a question of Petrunin \cite[\S 5]{Petrunin:veronese}. The proof of Theorem \ref{theo:norm-curv-sphere-thm} relies on the deep result by Brendle--Schoen \cite{BrendleSchoen} that a closed, strictly PIC-2 manifold is diffeomorphic to a standard spherical space form.

\subsection{General Veronese tensor products} 

The map in \eqref{eq:SnS1map} is part of a more general family of maps formed by taking tensor products of ``Veronese maps'' from spheres (immersions of spheres formed by normalizing an orthonormal basis of an eigenspace of the Laplacian). By optimizing within this class we find other immersions with small normal curvature including: 
\begin{enumerate}
\item $\cC_N(S^2\times S^2) \leq \sqrt{\frac 53}$ for $N\geq 14$ (see Proposition \ref{prop:prod-two-spheres}). 
\item $\cC_N(S^2\times T^2) \leq \sqrt{\frac 95}$ for $N \geq 24$ (see Proposition \ref{prop:Sn-T2}). 
\end{enumerate}
Whether these maps achieve the minimal curvature is an intriguing question. We pose possible intrinsic curvature conditions that imply the minimality. See Section \ref{sec:conj} for further discussion.

\subsection{On the use of AI} We began working on this problem after an internal DeepMind AI model (cf.\ \cite{Alethea}) produced the map 
\begin{equation}\label{eq:SnS1map}
f : S^n \times S^1 \ni (x,t) \mapsto (r_1 \cos t \, x, r_1 \sin t \, x, r_2 \cos 2t, r_2 \sin 2t) \in \RR^{2n+4},
\end{equation}
where $r_1 = \sqrt\frac 23 , r_2=\sqrt\frac13$ and proved that  $\fc(f) = \sqrt\frac32$ (in response to a request for low curvature immersions from $S^2\times S^1$). We found this exciting and it inspired us to consider this problem. This also lead us to consider the general ansatz discussed in Section \ref{sec:examples}. While finishing a draft of this article, we discovered that Naitoh had previously found \cite{Naitoh} a similar (but not identical) immersion with the same bound as in \eqref{eq:SnS1map}. As such, we cannot give the model credit for discovering \eqref{eq:SnS1map}, but do note that we would have been extremely unlikely to find the reference \cite{Naitoh} (or even to begin working on this problem) without the model's assistance.

We also found AI assistance useful in other ways. The DeepMind model found the curvature condition in Conjecture \ref{conj:imo-ric} in response to a request to ``reverse engineer'' a curvature condition based on considerations of the model case.  We also made use of publicly available LLMs (Gemini 2.5 and 3 Deep Think as well as ChatGPT 5.2 Pro). These models provided rigorous and semi-rigorous reasoning and numerical code generation that helped us find the maps in Section \ref{subsec:measures} (although this took significant iteration). These models were also able to help with routine but tedious proofs; namely, they helped us find proofs of optimizations involving second fundamental form bounds (cf.\ Lemma \ref{lemm:PIC2-helper-lemma}) and found the correct generalization Proposition \ref{prop:PIC2-petrunin} from PIC to PIC-2 by inserting ``$\lambda$'' and ``$\mu$'' as appropriate.

This article does not contain AI generated text.

\subsection{Acknowledgements}
O.C. was partially supported by a Terman Fellowship and an NSF grant (DMS-2304432). He thanks Google DeepMind for generously providing access to their models and particularly Tony Feng for facilitating this access as well for several interesting discussions.  C.L. was partially supported by an NSF grant (DMS-2202343) and a Sloan Fellowship.

\section{Preliminaries} 

\subsection{The Gauss equation} We take the curvature convention so that if $X \to \RR^N$ is an immersion with second fundamental form $\bA(\cdot,\cdot)$, the induced metric on $X$ has Riemann curvature tensor satisfying 
\begin{equation}\label{eq:Gauss}
\Rm(e_i,e_j,e_k,e_l) = \bangle{\bA(e_i,e_l),\bA(e_j,e_k)} - \bangle{\bA(e_i,e_k),\bA(e_j,e_l)}  
\end{equation}
for orthonormal tangent vectors $\{e_i, e_j, e_k, e_l\}$. In particular, in this convention, $\Rm(e_i,e_j,e_j,e_i)$ is a sectional curvature of $X$. Taking trace, we have that the Ricci curvature of the induced metric on $X$ satisfies
\begin{equation}\label{eq:tracedGauss}
\Ric(e_i,e_i) = \bangle{\bA(e_i,e_i),\bH} - |\bA(e_i,\cdot)|^2. 
\end{equation}
It is convenient to also recall Petrunin's version \cite[2.1]{Petrunin:tori} of the traced Gauss equations for an immersion $X^n\to\RR^N$
\begin{equation}\label{eq:PetruninGauss}
R = \frac 32 |\bH|^2 - \frac{n(n+2)}{2} \fint_{u \in S_pX} |\bA(u,u)|^2 
\end{equation}
where $R$ is the scalar curvature, $\bH = \tr\bA$ is the mean curvature, and $S_pX$ is the unit sphere in $T_pX$ with respect to the pullback metric.  

\subsection{Conformal change of curvature} We recall also some standard formulae for the behavior of curvature under a conformal change. For $(X^n,g)$ a Riemannian manifold and $\tilde g= e^{2\psi}g$ for $\psi \in C^\infty(X)$. Then the Riemann and Ricci curvature tensors with respect to $\tilde g$ satisfy
\begin{align}
\widetilde{\Rm} & = e^{2\psi} \left( \Rm - (\nabla^2 \psi)   \KN g + (d\psi\otimes d\psi) \KN g - \frac 12 |\nabla \psi|^2 g\KN g\right) \label{eq:confRm}\\
\widetilde{\Ric} & = \Ric - (n-2) (\nabla^2 \psi) + (n-2)(d\psi\otimes d\psi) - (\Delta \psi + (n-2) |\nabla \psi|^2)g,\label{eq:confRic}
\end{align}
where $\KN$ is the Kulkarni–Nomizu product defined by
\[
(h \KN k)(w,x,y,z) = h(w,z) k(x,y) + h(x,y)k(w,z) - h(w,y)k(x,z) - h(x,z) k(w,y). 
\]
See, e.g. \cite[Theorem 7.30]{Lee:Riemannian}.

\section{Immersions with small normal curvature}\label{sec:examples}

In this section we use tensor products of higher Veronese immersions to construct low curvature immersions of products of spheres. We note that the construction here is closely related to the tensor product construction of submanifolds of finite type as considered by B.-Y.\ Chen in \cite{Chen:book,Chen:tensor}, although it seems like this has not been considered in the context of \eqref{eq:grom-invar}. 

\subsection{Higher Veronese immersions} \label{subsec:Veronese}

We begin by recalling the notion of higher Veronese immersions (often called the \emph{standard minimal immersion} $S^n\to S^m$). The calculations here are well-known, cf.\ \cite{Takahashi:MinimalImmersion,DoCarmoWallach,Tsukada} but we include them for completeness. 

Recall that the spectrum of the Laplacian on the round sphere $(\SS^n,g_{\SS^n})$ is $\ell(\ell+n-1)$, $\ell=0,1,2\dots$ and the corresponding eigenspaces $E(n,\ell)$ have dimension $D(n,\ell) : = \binom{n+\ell-1}{n-1} + \binom{n+\ell-2}{n-1}$. Fix $n,\ell$, an $L^2$-orthonormal eigenbasis $f_1,\dots,f_D$ and then define $\varphi = \varphi_{n,\ell} : S^{n} \to \RR^{D(n,\ell)}$ by $\varphi = (|\SS^n|/D)^{\frac 12} (f_1,\dots,f_D)$. Since an isometry of $\SS^n$ induces a linear isometry of $E(n,\ell)$ it's clear that $|\varphi|^2 = C$ is constant. The given normalization ensures that $C=1$. Indeed, we may compute
\[
C |\SS^n| = \int_{\SS^n} |\varphi|^2 = |\SS^n| D^{-1} \sum_{i=1}^D \int_{\SS^n} |f_i|^2 = |\SS^n|. 
\]
Similar reasoning implies that $\varphi^*g_{\RR^D}$ is a round metric. Write
\[
\varphi^*g_{\RR^D} = D^{-1}|\SS^n|\sum_{i=1}^D df_i \otimes df_i = \rho \, g_{\SS^n}
\]
for $\rho=\rho(n,\ell) > 0$. Then 
\begin{align*}
n \rho D = \sum_{i=1}^D \int_{\SS^n} |df_i|^2 = \ell(\ell+n-1) \sum_{i=1}^D \int_{\SS^n} |f_i|^2 =  \ell(\ell+n-1) D
\end{align*}
so we find that
\begin{equation}\label{eq:rho-val-gen-ver}
\rho(n,\ell) = \frac{\ell(\ell+n-1)}{n} . 
\end{equation}
We now turn to the extrinsic curvature of $\varphi$. Since the coefficients of $\varphi$ are formed from an eigenbasis, we have that $\varphi$ is a minimal immersion into $\SS^{D-1}$. Thus, its Euclidean mean curvature satisfies $|\bH| =  n$.  As above, symmetry implies that the immersion is \emph{isotropic} in the sense that $|\bA(u,u)|$ is constant for all unit tangent vectors $u$. We define $\lambda=\lambda(n,\ell)$ by $|\bA(u,u)|^2 = \lambda |u|_{g_{\SS^n}}^4$ (note that the right hand side is \emph{not} the pullback metric!). Note that \eqref{eq:rho-val-gen-ver} gives that the pullback metric $\rho\, g_{\SS^n}$ has scalar curvature $\frac{n(n-1)}{\rho}$. Using these facts in \eqref{eq:PetruninGauss} (taking care to use the geometric scaling for the norm of $u$) we find
\begin{equation}\label{eq:lambda-rho-val-gen-ver}
\lambda(n,\ell) = \frac{3n} {n+2} \rho(n,\ell)^2 -  \frac{2(n-1)} {(n+2)}\rho(n,\ell). 
\end{equation}

\subsection{Tensor products of Veronese immersions} We now fix $M \in \NN$ and $\vec{n} \in \NN^M$. We set $X = \prod_{m=1}^M S^{n_m}$. For $\vec{\ell} \in \NN_0^M$ we set $L = L(\vec{n},\vec{\ell}) = \prod_{m=1}^M D(n_m,\ell_m)$ and then define the tensor product maps
\[
\varphi_{\vec{\ell}} : X \to \RR^L, \qquad (x_1,\dots,x_M) \mapsto \otimes_{m=1}^M \varphi_{n_m,\ell_m}(x_m). 
\]
Note that if $\ell_m = 0$ our convention is that $\varphi_{n_m,0} = 1 : S^{n_m} \to \RR$ is a constant map. In particular, $\varphi_{\vec{\ell}}$ will not be an immersion in this case. 

Using 
\[
d\varphi_{\vec\ell} = \sum_{m=1}^M \varphi_{n_1,\ell_1} \otimes \cdots \otimes d\varphi_{n_m,\ell_m} \otimes \cdots \otimes \varphi_{n_M,\ell_M}
\]
along with $d\varphi_{n,\ell}\perp \varphi_{n,\ell}=0$ and $|\varphi_{n,\ell}|=1$ we find
\[
(\varphi_{\vec\ell})^*g_{\RR^L}  = (\rho(n_1,\ell_1)g_{\SS^{n_1}},\dots,\rho(n_M,\ell_M)g_{\SS^{n_M}}).
\]

We now consider a geodesic $\gamma(t) = (\gamma_1(t),\dots,\gamma_M(t))$ with respect to the product metric $(g_{\SS^{n_1}},\dots,g_{\SS^{n_M}})$. Let $(x_1,\dots,x_M) = \gamma(0)$ and $(u_1,\dots,u_M) = \gamma'(0)$. Note that we are \emph{not} assuming that $\gamma$ has unit speed. We compute
\begin{align*}
& \frac{d^2}{dt^2} \Big|_{t=0} \varphi_{\vec{\ell}}\, (\gamma(t)) \\
& = \sum_{m=1}^M \varphi_{n_1,\ell_1}(x_1) \otimes \dots\otimes \bA_{\varphi_{n_m,\ell_m}}(u_m,u_m) \otimes \dots\otimes \varphi_{n_M,\ell_M}(x_M)\\
& + 2 \sum_{a<b} \varphi_{n_1,\ell_1}(x_1)\otimes \dots \otimes d\varphi_{n_a,\ell_a}(u_a) \otimes \dots\otimes  d\varphi_{n_b,\ell_b}(u_b) \otimes \dots \otimes \varphi_{n_M,\ell_M}(x_M). 
\end{align*}
Observe that the inner product of the two sums vanishes, but distinct terms in the first sum are not pairwise orthogonal since for $a\neq b$ we have (writing $\varphi_a$ for $\varphi_{n_a,\ell_a}$ and so on)
\[
\bangle{ \cdots\otimes \bA_{\varphi_{a}}(u_a,u_a) \otimes \cdots,\cdots\otimes \bA_{\varphi_{b}}(u_b,u_b) \otimes \cdots } = \rho(n_a,\ell_a) \rho(n_b,\ell_b) |u_a|^2|u_b|^2
\]
where we used $\bangle{\bA_{\varphi_a}(u_a,u_a),\varphi_a} = -|d\varphi_a(u_a)|^2 = - \rho(n_a,\ell_a)|u_a|^2$. Again, we emphasize that here we have written $|\cdot|$ for the norm with respect to $g_{\SS^n}$, not the pullback metric. We use this convention below (note that we also write the norm of a vector in $\RR^L$, where we unambiguously use the Euclidean inner product). Putting these together, we find
\[
\left| \frac{d^2}{dt^2} \Big|_{t=0} \varphi_{\vec{\ell}}\, (\gamma(t)) \right|^2 = \sum_{m=1}^M \lambda(n_m,\ell_m)|u_m|^4 + 6 \sum_{a<b} \rho(n_a,\ell_a)\rho(n_b,\ell_b) |u_a|^2 |u_b|^2.
\]

Now, we fix $\alpha : \NN_0^M \to [0,\infty)$ with compact support and 
\[
\sum_{\vec{\ell} \in\NN_0^M} \alpha(\vec \ell) = 1.
\]
It will be convenient to identify $\alpha$ with the associated compactly supported probability measure $\sum_{\vec{\ell} \in\NN_0^M} \alpha(\vec\ell) \delta_{\vec{\ell}}$ on $\NN_0^M$. For any such measure, $\alpha(\vec{\ell})$, we define
\begin{equation}\label{eq:defn-F-tensor-sum-veronese}
F : X \to \RR^N, \qquad F(x) = \bigoplus_{\vec{\ell} \in \supp \alpha} \alpha(\vec{\ell})^{\frac 12} \varphi_{\vec{\ell}}\, (x),
\end{equation}
for $N = \sum_{\vec{\ell} \in \supp \alpha} L(\vec{n},\vec{\ell})$. Note that $|F(x)| = 1$ for any $x \in X$, so the image of $F$ lies in $\BB^N$ (actually $\partial \BB^N$). Moreover,
\[
|u|_{F^*g_{\RR^N}}^2 = \sum_{m=1}^M \left( \sum_{\vec{\ell} \in \NN_0^M} \alpha(\vec{\ell}) \rho(n_m,\ell_m) \right)|u_m|^2 = \sum_{m=1}^M \EE[\rho_m] |u_m|^2
\]
where $\rho_m : \NN_0^M \to [0,\infty)$ is $\rho_m(\vec{\ell}) = \rho(n_m,\ell_m)$ and $\EE[\cdot]$ is the expectation with respect to the probability measure $\mu$. In particular we find that $F$ is an immersion if and only if $\EE[\rho_m] > 0$ for $m=1,\dots,M$. Similarly we may compute 
\begin{align*}
|\bA_F(u,u)|^2 = \sum_{m=1}^M \EE[\lambda_m]|u_m|^4 + 6 \sum_{a<b} \EE[\rho_a\rho_b]\, |u_a|^2|u_b|^2
\end{align*} 
where $\lambda_m(\ell) = \lambda(n_m,\ell)$, and by definition,
\[\EE [\lambda_m] = \sum_{\vec\ell \in \NN_0^M} \alpha(\vec\ell) \lambda(n_m, l_m), \quad \EE [\rho_a\rho_b] = \sum_{\vec{\ell} \in \NN_0^M} \alpha(\ell) \rho(n_a, \ell_a) \rho(n_b, \ell_b).\]

Let $A$ be the symmetric $M\times M$ matrix with diagonal coefficients $A_{mm} = \EE[\lambda_m]$ and off diagonal coefficients $A_{ab} =3\EE[\rho_a\rho_b]$. Define $G \in \RR^M$ by $G_m = \EE[\rho_m]$. Finally, define $U \in \RR^M$ by $U_m = |u_m|^2$. Note that $U$ lies in the first quadrant: $U \in \cU := \{U_m \geq 0\}\setminus\{0\}$. The previous analysis thus gives that the normal curvature of $F$ satisfies
\[
\fc(F)^2 = \max_{U \in \cU} \frac{U^T AU}{U^T GG^TU}. 
\]
In particular, we have proven:
\begin{prop}\label{prop:Bs-copos-reformulation}
Fix a probability measure $\mu = \sum_{\vec{\ell} \in\NN_0^M}\alpha(\vec{\ell})\delta_{\vec{\ell}}$ with compact support. Assume that $\EE[\rho_m] > 0$ for $m=1,\dots,M$. Then the map $F : X \to \RR^N$ defined in \eqref{eq:defn-F-tensor-sum-veronese} is an immersion into $\BB^N$. Let $s \geq 0$ be minimal so that 
\[
B(s) := sGG^T - A
\] 
is copositive. Then $\fc(F) = \sqrt{s}$. 
\end{prop}
Recall that a (symmetric) matrix $B$ is \emph{copositive} if $U^TBU \geq 0$ for all $U \in \cU$. 

\begin{rema}
Note that $F$ is \emph{isotropic} (in the sense that $|\bA(u,u)|$ is constant for all unit tangent vectors $u$) if and only if $B(s) = 0$. This will be the case in the examples below. 
\end{rema}

\subsection{Low curvature examples} \label{subsec:measures} We consider small $M$ and obtain measures $\mu$ corresponding to low curvature maps (i.e.\ the parameter $s$ is small). The following is a direct generalization of \eqref{eq:SnS1map}. 
\begin{prop}\label{prop:prod-two-spheres}
Consider $n_1 \geq n_2$. We have $\cC_N(S^{n_1} \times S^{n_2}) \leq \sqrt{\frac{2n_2+1}{n_2+1}}$ for $N \geq (n_1+1)(n_2+1) + \frac{n_2(n_2+3)}{2}$. 
\end{prop} 
\begin{proof}
We consider
\[
\mu = \alpha(1,1) \, \delta_{(1,1)} + \alpha(0,2) \,  \delta_{(0,2)}. 
\]
Recall that $\rho(n,0) = 0, \rho(n,1) = 1,\rho(n,2) = \frac{2(n+1)}{n}$ and $\lambda(n,0) = 0, \lambda(n,1) = 1,\lambda(n,2) = \frac{8(n+1)}{n}$. Then $G = (\alpha(1,1),\alpha(1,1) + \frac{2(n_2+1)}{n_2} \alpha(0,2))^T$ and
\[
A = \left( \begin{matrix} \alpha(1,1) & 3 \alpha(1,1) \\ 3\alpha(1,1) & \alpha(1,1) + \frac{8(n_2+1)}{n_2} \alpha(0,2)  \end{matrix} \right)
\]
We seek $\alpha(1,1),\alpha(0,2),s$ so that $B(s) = 0$. This gives
\begin{align*}
s \alpha(1,1)^2 & = \alpha(1,1)\\
s\alpha(1,1) (\alpha(1,1) + \tfrac{2(n_2+1)}{n_2} \alpha(0,2)) & = 3 \alpha(1,1)\\
s(\alpha(1,1) + \tfrac{2(n_2+1)}{n_2} \alpha(0,2))^2 & = \alpha(1,1) + \tfrac{8(n_2+1)}{n_2} \alpha(0,2).
\end{align*}
We can solve this system to find $\alpha(1,1) = \frac{n_2+1}{2n_2+1}$, $\alpha(0,2) = \frac{n_2}{2n_2+1}$, and $s=\frac{2n_2+1}{n_2+1}$.

One may check that with these choices, $\EE[\rho_m] > 0$ for $m=1,2$ so the corresponding map is an immersion, and that $B(s) = 0$.  Thus, the claim follows from Proposition \ref{prop:Bs-copos-reformulation}. 
\end{proof}

\begin{prop}
For $n_1\geq n_2\geq n_3$ we have $\cC_N(S^{n_1}\times S^{n_2}\times S^{n_3}) \leq \sqrt{\frac{6n_3+5}{3n_3+3}}$ for $N\geq (n_1+1)(n_2+1) + (n_1+1)(n_3+1) + (n_2+1)(n_3+1) + \frac{n_3(n_3+3)}{2}$.  
\end{prop}
\begin{proof}
Trial and error yield the following generalization of the measure from Proposition \ref{prop:prod-two-spheres}
\[
\mu = \frac{n_3+1}{6n_3+5} \delta_{(1,1,0)} + \frac{2(n_3+1)}{6n_3+5} (\delta_{(0,1,1)} + \delta_{(1,0,1)}) + \frac{n_3}{6n_3+5} \delta_{(0,0,2)}.
\]
One may check that $B(s) = 0$ with $s=\frac{6n_3+5}{3n_3+3}$ and conclude using Proposition \ref{prop:Bs-copos-reformulation} as before. 
\end{proof}
We are able to improve the previous estimate when $n_2=n_3 = 1$ as follows:
\begin{prop}\label{prop:Sn-T2}
For $n\geq 1$ we have $\cC_N(S^n \times T^2) \leq \sqrt{\frac95}$ for $N \geq 4n+16$. 
\end{prop}
\begin{proof}
Let
\[
\mu = \frac{5}{9} \delta_{(1,5,5)} + \frac{200}{7371} \delta_{(0,2,11)} + \frac{719}{3024} \delta_{(0,5,10)} + \frac{3025}{16848}\delta_{(0,11,2)}. 
\]
One can check that $\mu$ is a probability measure. Recalling that $\rho(1,\ell)=\ell^2,\lambda(1,\ell)=\ell^4$ and $\rho(n,1) = \lambda(n,1) = 1$, we compute
\begin{align*}
G_1 & = \frac 5 9\\
G_2 = G_3 & = \frac{125}{3}\\
A_{11} & = \frac 59\\
A_{12} = A_{13} & = \frac{125}{3} \\
A_{22} & = 3125\\
A_{23} & = 3125\\
A_{33} & = 3125. 
\end{align*}
From this, we find that $B(\frac 95) = 0$. Using Proposition \ref{prop:Bs-copos-reformulation}, this completes the proof. 
\end{proof}

Observe that any choice of $\mu$ induces a product metric on $S^n\times T^2$ and thus contains a totally geodesic $T^3$. Since $\cC_N(T^3) \geq \sqrt{\frac9 5}$, this is the optimal bound that can be obtained within this framework of tensors of higher Veronese immersions (but note that this does not prove optimality among all immersions). The measure in Proposition \ref{prop:Sn-T2} was found by a combination of numerical search and repeated LLM queries.

\begin{rema}
We may recover Gromov's bound $\cC_N(T^n) \leq \sqrt{\frac{3n}{n+2}}$ for $N \gg n$ using this framework. Note that $\rho(1,\ell) = \ell^2$ and $\lambda(1,\ell) = \ell^4$. Fix a rational spherical $4$-design $D = \{\hat u_1,\dots,\hat u_K\} \subset \SS^{n-1}$ (cf.\ \cite[Section 13]{Gromov:notes}). This means that $\hat u_j$ is a rational point on $S^{n-1}$ so that
\[
\frac 1 K \sum_{k=1}^K p(\hat u_k) = \fint_{\SS^{n-1}} p(x) \, dx
\]
for any polynomial $p$ on $\RR^{n}$ of degree $\leq 4$. In particular, we get
\begin{align*}
\frac{1}{K} \sum_{k=1}^K \bangle{\hat u_k,e_m}^2 & = \frac 1 n \\ 
\frac{1}{K} \sum_{k=1}^K \bangle{\hat u_k,e_m}^4 & = \frac{3}{n(n+2)} \\ 
\frac{1}{K} \sum_{k=1}^K \bangle{\hat u_k,e_a}^2\bangle{\hat u_k,e_b}^2 & = \frac{1}{n(n+2)}, 
\end{align*}
with $a\neq b$. Since $\hat u_k \in \QQ^{n}$ we can find $Q \in \NN$ with $u_k : = Q \hat u_k \in \ZZ^n$. Then set
\[
\mu = \frac 1 K \sum_{k=1}^K \delta_{u_k}. 
\]
We have
\begin{align*}
\EE[\ell_m^2] & = \frac{Q^2}{n}\\
\EE[\ell_m^4] & = \frac{3Q^4}{n(n+2)}
\end{align*}
so $G_m = \frac {Q^2}{n}$ and $A$ has on-diagonal elements $\EE[\ell_m^4] = \frac{3Q^4}{n(n+2)}$ and off-diagonal $3\EE[\ell_a^2\ell_b^2] = \frac{3Q^4}{n(n+2)}$. As such, we see that $B(\frac{3n}{n+2}) = 0$ as claimed. 
\end{rema}

\section{Obstructions to low curvature maps}

\subsection{A generalization of Petrunin's angle estimate}
The following lemma is a generalization of \cite[Lemma 4.1]{Petrunin:tori}.  

\begin{lemm}\label{lemm:angle}
Suppose that for $X$ closed, $f : X \to \BB^N$ has $\fc(f) \leq c <  2$.  For $p \in X$ write $\bx = f(p)$ and $\bx^\perp$ for the projection of $\bx$ to the normal bundle to $X$ at $p$. Then $|\bx| > 0$ and $|\bx^\perp| \geq 1 + \frac c2 (|\bx|^2-1)$. 
\end{lemm}
\begin{proof}
Fix $p \in X$ and set $\bx = f(p)$. We first show that $\bx \neq 0$. If $\bx = 0$, then fix any unit $\xi \in T_p X$ and let $\gamma(t) = f(\exp_p(t\xi))$ be the image of the geodesic in $X$ (with respect to the pullback metric) through $p$ with initial velocity $\xi$. Because $\gamma$ has curvature $\leq c$, the bow lemma (comparison with a circular arc) implies that $|\gamma(\frac \pi c)| \geq \frac 2 c >1$, a contradiction. Thus, $|\bx|\neq 0$. This proves the first part of the lemma. 

We now decompose $\bx = \bx^\perp+\bx^\top$. If $\bx^\top \neq 0$ then set $\xi = \frac{\bx^\top}{|\bx^\top|}$. Otherwise, choose any $\xi \in T_pX$. If $\bx^\perp \neq 0$ then set $\eta = \frac{\bx^\perp}{|\bx^\perp|}$. Otherwise, choose any unit $\eta$ perpendicular to $\xi$. Let $\Pi = \Span\{\eta,\xi\}$ and note that $\bx \in \Pi$.

Let $\by = \bx - c^{-1}\eta$ and set $C= \{\bz \in \Pi : |\bz - \by| = c^{-1}\}$, a circle of radius $c^{-1}$ passing through $\bx$, tangent to $\xi$ there. Observe that if the assertion failed, i.e.\
\begin{equation}\label{eq:angle-contr-assump}
|\bx^\perp| < 1 + \frac c2 (|\bx|^2 -1)
\end{equation}
then $C$ is not contained in the unit ball at the origin. Indeed,  this would give
\[
|\by|^2 = |\bx|^2 + \frac{1}{c^2} - \frac 2c \bangle{\bx,\eta} =  |\bx|^2 + \frac{1}{c^2} - \frac 2c |\bx^\perp| >  \frac{1}{c^2} - \frac{2}{c} +1 = \left( 1 - \frac 1 c\right)^2. 
\]
At this point, we may conclude the proof exactly as in \cite[Lemma 4.1]{Petrunin:tori}. 

Indeed, if \eqref{eq:angle-contr-assump} held, we may form a comparison curve $\sigma * \gamma$ using a unit speed circular arc $\sigma$ from the origin that meets $\bx$ with tangent $\xi$. We can do this, so that $\sigma$ may be concatenated with a parametrization of a portion of $C$, denoted by $\gamma$, in a $C^1$ convex manner, so that $\gamma(t_0) \not \in \BB^N$. For $\tilde\gamma$ the image under $f$ of the geodesic in $X$ from $p$ with initial velocity $\xi$, the bow lemma (for $\sigma * \gamma$ versus $\sigma * \tilde\gamma$; cf.\ \cite[Lemma 3.19]{PetruninBarrera}) implies that $\tilde\gamma$ exits $\BB^N$. This is a contradiction. 
\end{proof}

\subsection{Conformally positive sectional curvature} 
In this section we prove Proposition \ref{prop:secPetrunin}. Consider $f : X^n \to \BB^N$ with $n\geq 2$. Later we will assume $\fc(f) < \sqrt\frac 32$. 

\begin{lemm}\label{lemm:sff-bds-off-diag}
For $i \neq j$, we have $|\bA(e_i,e_i) + \bA(e_j,e_j)|^2 + 4 |\bA(e_i,e_j)|^2 \leq 4 \fc(f)^2$. 
\end{lemm}
\begin{proof}
Take $u_\pm = \frac{1}{\sqrt{2}}(e_i \pm e_j)$ in \eqref{eq:defn-nc} to get $|\bA(e_i,e_i) + \bA(e_j,e_j) \pm 2 \bA(e_i,e_j)|^2 \leq 4 \fc(f)^2$. The assertion then follows from the parallelogram identity. 
\end{proof}

We are now ready for the proof of conformal positivity of sectional curvature. 

\begin{proof}[Proof of Proposition \ref{prop:secPetrunin}]
Let $g$ denote the induced metric and consider $\psi=\psi(r)$ for $r=|\bx|$ the ambient distance to the origin. We compute
\begin{align*}
\nabla \psi & = \psi'(r) \frac{\bx^\top}{r} \\
|\nabla \psi|^2 & = \frac{\psi'(r)^2}{r^2}|\bx^\top|^2\\
\nabla^2 r^2 & = 2 g + 2\bangle{\bx^\perp , \bA }\\
\nabla^2 \psi & = \left(\frac{\psi''(r)}{r^2} - \frac{\psi'(r)}{r^3} \right) \bx^\top\otimes \bx^\top + \frac{\psi'(r)}{r} g + \frac{\psi'(r)}{r} \bangle{\bx^\perp,\bA}
\end{align*}
Let $\tilde g= e^{2\psi}g$. For $e_1,\dots,e_n$ a $g$-orthonormal basis, fix the corresponding $\tilde g$-orthonormal basis $\tilde e_i = e^{-\psi}e_i$. 

Then, for $i\neq j$, the conformal change of curvature formula \eqref{eq:confRm} and the Gauss equations \eqref{eq:Gauss} gives
\begin{align*}
& e^{2\psi} \tilde \sec_{ij} \\
& = e^{2\psi}  \widetilde\Rm(\tilde e_i,\tilde e_j, \tilde e_j, \tilde e_i) \\
& = \Rm(e_i,e_j,e_j,e_i) - \nabla^2 \psi(e_i,e_i) - \nabla^2 \psi(e_j,e_j) + (d\psi(e_i))^2 + (d\psi(e_j))^2 - |\nabla \psi|^2\\
& =  - 2 \frac{\psi'(r)}{r} +  \bangle{\bA(e_i,e_i),\bA(e_j,e_j)} - |\bA(e_i,e_j)|^2 - \frac{\psi'(r)}{r} \bangle{\bx^\perp,\bA(e_i,e_i)+\bA(e_j,e_j)} \\
& \qquad - \left( \frac{\psi''(r)}{r^2} - \frac{\psi'(r)}{r^3} - \frac{\psi'(r)^2}{r^2} \right)(\bangle{\bx,e_i}^2 + \bangle{\bx,e_j}^2) - \frac{\psi'(r)^2}{r^2} |\bx^\top|^2 . 
\end{align*}
We take $\psi(r) = - \frac c2 r^2$ for $c>0$ to be chosen below. This gives
\begin{align*}
e^{2\psi} \tilde \sec_{ij} & = 2c +  \bangle{\bA(e_i,e_i),\bA(e_j,e_j)} - |\bA(e_i,e_j)|^2 + c \bangle{\bx^\perp,\bA(e_i,e_i)+\bA(e_j,e_j)} \\
& \qquad + c^2 (\bangle{\bx,e_i}^2 + \bangle{\bx,e_j}^2) - c^2 |\bx^\top|^2 \\
& \geq  2c + \bangle{\bA(e_i,e_i),\bA(e_j,e_j)} - |\bA(e_i,e_j)|^2 - c|\bx^\perp| |\bA(e_i,e_i)+\bA(e_j,e_j)| \\
& \qquad  - c^2 |\bx^\top|^2 \\
& \geq 2c + \bangle{\bA(e_i,e_i),\bA(e_j,e_j)} - |\bA(e_i,e_j)|^2 - \frac{c}{4}  |\bA(e_i,e_i)+\bA(e_j,e_j)|^2 \\
& \qquad - c|\bx^\perp|^2 - c^2 |\bx^\top|^2 \\
& \geq 2c - \fc(f)^2 + \bangle{\bA(e_i,e_i),\bA(e_j,e_j)} + \frac{1-c}{4}  |\bA(e_i,e_i)+\bA(e_j,e_j)|^2 \\
& \qquad - c|\bx^\perp|^2 - c^2 |\bx^\top|^2.
\end{align*}
We used Lemma \ref{lemm:sff-bds-off-diag} to bound the off-diagonal terms in the second fundamental form in the final inequality. By considering the model case \eqref{eq:SnS1map}, we see that in some cases the $\bangle{\bA(e_i,e_i),\bA(e_j,e_j)}$ terms cannot be controlled in a sharp manner. Thus, we must choose $c=3$ so they cancel. With this choice we have
\begin{align*}
& e^{2\psi} \widetilde {\sec}_{ij} \\
& \geq 2c - \fc(f)^2 + \frac{3-c}{2}\bangle{\bA(e_i,e_i),\bA(e_j,e_j)} + \frac{1-c}{4}  |\bA(e_i,e_i)|^2 + \frac{1-c}{4} |\bA(e_j,e_j)|^2 \\
& \qquad - c|\bx^\perp|^2 - c^2 |\bx^\top|^2\\
& = 6 - \fc(f)^2 - \frac{1}{2}  |\bA(e_i,e_i)|^2 - \frac{1}{2} |\bA(e_j,e_j)|^2 \\
& \qquad - 3 |\bx^\perp|^2 - 9 |\bx^\top|^2\\
& \geq 6 - 2 \fc(f)^2  - 3 |\bx^\perp|^2 - 9 |\bx^\top|^2\\
& \geq 6 - 2 \fc(f)^2  + 6 |\bx^\perp|^2 - 9 |\bx|^2. 
\end{align*}
Now, we assume that $\fc(f) \leq \sqrt{\frac 32}$. Lemma \ref{lemm:angle} gives 
\[
6 |\bx^\perp|^2 - 9 |\bx|^2 \geq 6 \left( 1+ \frac 12 \sqrt{\frac32} (|\bx|^2-1)\right)^2 - 9 |\bx|^2.
\]
This is a quadratic in $|\bx|^2$, and simple calculation proves the minimum of this expression on $|\bx| \in [0,1]$ is attained at $|\bx|=1$, so 
\[
6 |\bx^\perp|^2 - 9 |\bx|^2 \geq -3.
\]
Putting this together, we get 
\[
e^{2\psi} \widetilde{\sec}_{ij} \geq 3-2\fc (f)^2. 
\]
This proves Proposition \ref{prop:secPetrunin}. 
\end{proof}

\subsection{Conformally positive isotropic curvature}

In this section we use a version of the proof from the previous section to prove the following result. 
\begin{prop}\label{prop:PIC2-petrunin}
For $n\geq 4$, if $X^n$ is closed and admits an immersion $f : X \to\BB^N$ with $\fc(f) < \sqrt\frac 43$ then the induced metric $g = f^* g_{\RR^N}$ is conformally equivalent to $\tilde g$ that is strictly PIC-2 in the sense of \eqref{eq:tilde-PIC2} below. 
\end{prop}

We recall that the strict PIC-2 condition is
\begin{multline}\label{eq:tilde-PIC2}
\tilde \cQ_{\lambda,\mu}(\tilde e_1,\tilde e_2,\tilde e_3,\tilde e_4) : = \widetilde{\sec}(\tilde e_1,\tilde e_3) + \lambda^2 \widetilde{\sec}(\tilde e_1,\tilde e_4) + \mu^2 \widetilde{\sec}(\tilde e_2,\tilde e_3) \\ + \lambda^2 \mu^2 \widetilde{\sec}(\tilde e_2,\tilde e_4) - 2\lambda\mu \widetilde{\Rm}(\tilde e_1,\tilde e_2,\tilde e_3,\tilde e_4) > 0
\end{multline}
for $\tilde e_1,\dots,\tilde e_4$ any $\tilde g$-orthonormal set and $\lambda,\mu \in [-1,1]$. 

Before proving Proposition \ref{prop:PIC2-petrunin}, we deduce the following consequence. 
\begin{proof}[Proof of Theorem \ref{theo:norm-curv-sphere-thm} using Proposition \ref{prop:PIC2-petrunin}] 
Consider $X^n$ closed and an immersion $f : X \to\BB^N$ with $\fc(f) < \sqrt\frac 43$. By work of Petrunin \cite{Petrunin:veronese}, $X^n$ is homeomorphic to an $n$-sphere, and is thus simply connected. This already proves the assertion for $n \leq 3$. 

For $n\geq 4$, by combining Proposition \ref{prop:PIC2-petrunin} with the classification of strict PIC-2 manifolds (\emph{en route} to their proof of the differentiable sphere theorem) by Brendle--Schoen \cite[Theorem 3]{BrendleSchoen}, we have that $X$ is diffeomorphic to a (standard) spherical space form and thus a standard sphere by simple connectivity.
\end{proof}

Consider $X^n$ closed and an immersion $f : X^n \to \BB^N$. Later we will assume $\fc(f) < \sqrt\frac 43$. We fix  $p \in X$, with $\bx = f(p)$ and $e_1,\dots,e_4 \in T_p X$ an arbitrary orthonormal set (with respect to $g=f^*g_{\RR^N}$). (Of course, if we set $\tilde e_i = e^{-\psi} e_i$, $i=1,\dots,4$, this defines an arbitrary $\tilde g$-orthonormal set.) We also fix $\lambda, \mu \in [-1,1]$ and define
\begin{align*}
\bW & = \bA(e_1,e_1) + \mu^2 \bA(e_2,e_2) \\
\bX & = \bA(e_3,e_3) + \lambda^2 \bA(e_4,e_4)\\
\bY & = \bA(e_1,e_3) - \lambda \mu \bA(e_2,e_4)\\
\bZ & = \lambda \bA(e_1,e_4) + \mu \bA(e_2,e_3)\\
\bS & = (1+\lambda^2)\bW + (1+\mu^2)\bX.
\end{align*} 
The following will be used in a similar manner to Lemma \ref{lemm:sff-bds-off-diag}. 
\begin{lemm}\label{lemm:PIC2-helper-lemma}
We have
\begin{align*}
|\bS|^2 + 4 (1+\lambda^2)(1+\mu^2) |\bY|^2 & \leq 4  (1+\lambda^2)^2(1+\mu^2)^2 \fc(f)^2\\
|\bS|^2 + 4 (1+\lambda^2)(1+\mu^2) |\bZ|^2 & \leq 4  (1+\lambda^2)^2(1+\mu^2)^2 \fc(f)^2.
\end{align*}
\end{lemm}
\begin{proof}
Let
\[
u_{\pm} : = \sqrt{1+\lambda^2} \, e_1 \pm \sqrt{1+\mu^2} \, e_3, \qquad v_{\pm}  : = \mu \sqrt{1+\lambda^2} \, e_2 \mp  \lambda \sqrt{1+\mu^2} \, e_4. 
\]
Below, we will consider ``$\pm$'' to be fixed as one of $+$ or $-$, with the choice made consistently through the full expression. We have 
\begin{align*}
\bA(u_{\pm},u_{\pm}) & = (1+\lambda^2)\bA(e_1,e_1) + (1+\mu^2)\bA(e_3,e_3) \pm 2\sqrt{1+\lambda^2} \, \sqrt{1+\mu^2} \, \bA(e_1,e_3)\\
\bA(v_\pm,v_\pm) & = \mu^2(1+\lambda^2)\bA(e_2,e_2) + \lambda^2(1+\mu^2) \bA(e_4,e_4) \mp 2 \mu \lambda \sqrt{1+\lambda^2} \, \sqrt{1+\mu^2} \, \bA(e_2,e_4). 
\end{align*}
Thus
\begin{align*}
\bA(u_\pm,u_\pm) + \bA(v_\pm,v_\pm) = \bS \pm 2 \sqrt{1+\lambda^2} \, \sqrt{1+\mu^2}\, \bY. 
\end{align*}
On the other hand,
\begin{align*}
|\bA(u_\pm,u_\pm) + \bA(v_\pm,v_\pm) | & \leq |\bA(u_\pm,u_\pm)| + | \bA(v_\pm,v_\pm) |\\
& \leq \fc(f)\left( |u_\pm|^2 + |v_\pm|^2 \right)\\
& =2\fc(f)(1+\lambda^2)(1 + \mu^2)
\end{align*}
Combining these expressions with the parallelogram identity, the first inequality follows. 

For the second inequality, we set
\[
u_{\pm} : = \sqrt{1+\lambda^2} \, e_1 \pm \lambda \sqrt{1+\mu^2} \, e_4, \qquad v_{\pm}  : = \mu \sqrt{1+\lambda^2} \, e_2 \pm  \sqrt{1+\mu^2} \, e_3. 
\]
The inequality then follows by essentially same steps as used above. This completes the proof. 
\end{proof}

We now have the proof of the conformal PIC-2 property.
\begin{proof}
We argue as in the proof of Proposition \ref{prop:secPetrunin}. For $\psi = \psi(r) = -\frac c 2 r^2$, with $c>0$ to be chosen we set $\tilde g= e^{2\psi} g$. In Proposition \ref{prop:secPetrunin} we computed 
\begin{align*}
& e^{2\psi} \widetilde{\Rm}(\tilde e_i,\tilde e_j,\tilde e_j,\tilde e_i) 
\\
& \geq 2c +  \bangle{\bA(e_i,e_i),\bA(e_j,e_j)} - |\bA(e_i,e_j)|^2 + c \bangle{\bx^\perp,\bA(e_i,e_i)+\bA(e_j,e_j)} \\
& \qquad  - c^2 |\bx^\top|^2.
\end{align*}
Moreover, \eqref{eq:confRm} and \eqref{eq:Gauss} give that 
\begin{align*}
e^{2\psi} \widetilde{\Rm}(\tilde e_1,\tilde e_2,\tilde e_3,\tilde e_4) & = \Rm(e_1,e_2,e_3,e_4) \\
& = \bangle{\bA(e_1,e_4),\bA(e_2,e_3)} - \bangle{\bA(e_1,e_3),\bA(e_2,e_4)}. 
\end{align*}
Thus (recalling the definition in \eqref{eq:tilde-PIC2}) we have
\begin{align*}
& e^{2\psi} \tilde \cQ_{\lambda,\mu}(\tilde e_1,\tilde e_2,\tilde e_3,\tilde e_4) \\
& \geq 2c(1+\lambda^2)(1 + \mu^2)  + \bangle{ \bA(e_1,e_1) + \mu^2 \bA(e_2,e_2) ,\bA(e_3,e_3) + \lambda^2 \bA(e_4,e_4)} \\
& \qquad - |\bA(e_1,e_3)|^2 - \lambda^2 |\bA(e_1,e_4)|^2 - \mu^2 |\bA(e_2,e_3)|^2 - \lambda^2\mu^2|\bA(e_2,e_4)|^2\\
& \qquad -2 \lambda \mu \bangle{\bA(e_1,e_4),\bA(e_2,e_3)} + 2\lambda \mu \bangle{\bA(e_1,e_3),\bA(e_2,e_4)}\\
& \qquad + c (1+\lambda^2) \bangle{\bx^\perp, \bA(e_1,e_1) + \mu^2 \bA(e_2,e_2) }\\
& \qquad + c (1+\mu^2) \bangle{\bx^\perp,  \bA(e_3,e_3) + \lambda^2 \bA(e_4,e_4) }\\
& \qquad - c^2 (1+\lambda^2)(1+\mu^2) |\bx^\top|^2\\
& = 2c(1+\lambda^2)(1+\mu^2) + \bangle{\bW,\bX} - |\bY|^2 - |\bZ|^2  + c\bangle{\bx^\perp,\bS}
- c^2(1+\lambda^2)(1+\mu^2) |\bx^\top|^2\\
& \geq 2c(1+\lambda^2)(1+\mu^2) + \bangle{\bW,\bX} - |\bY|^2 - |\bZ|^2  - \frac{c}{4(1+\lambda^2)(1+\mu^2)} |\bS|^2 \\
& \qquad - c(1+\lambda^2)(1+\mu^2) |\bx^\perp|^2 - c^2(1+\lambda^2)(1+\mu^2)|\bx^\top|^2.
\end{align*}
Summing the inequalities in Lemma \ref{lemm:PIC2-helper-lemma} gives
\[
\frac{1}{2(1+\lambda^2)(1+\mu^2)} |\bS|^2 +  |\bY|^2 + |\bZ|^2  \leq   2(1+\lambda^2)(1+\mu^2) \fc(f)^2
\]
Thus we obtain
\begin{align*}
e^{2\psi} \tilde \cQ_{\lambda,\mu}(\tilde e_1,\tilde e_2,\tilde e_3,\tilde e_4) & \geq 2(1+\lambda^2)(1+\mu^2)(c - \fc(f)^2)\\
& \qquad  + \bangle{\bW,\bX}  + \frac{2-c}{4(1+\lambda^2)(1+\mu^2)} |\bS|^2 \\
& \qquad - c(1+\lambda^2)(1+\mu^2) |\bx^\perp|^2 - c^2(1+\lambda^2)(1+\mu^2)|\bx^\top|^2.
\end{align*}
Observing that
\[
\frac{|\bS|^2}{2(1+\lambda^2)(1+\mu^2)} = \frac{1+\lambda^2}{2(1+\mu^2)} |\bW|^2 + \frac{1+\mu^2}{2(1+\lambda^2)} |\bX|^2 + \bangle{\bW,\bX},
\]
we thus choose $c=4$. Putting this together, we have 
\begin{align*}
e^{2\psi} \tilde \cQ_{\lambda,\mu}(\tilde e_1,\tilde e_2,\tilde e_3,\tilde e_4) & \geq 2(1+\lambda^2)(1+\mu^2)(4 - \fc(f)^2)\\
& \qquad - \frac{1+\lambda^2}{2(1+\mu^2)} |\bW|^2 - \frac{1+\mu^2}{2(1+\lambda^2)} |\bX|^2  \\
& \qquad - 4(1+\lambda^2)(1+\mu^2) |\bx^\perp|^2 - 16(1+\lambda^2)(1+\mu^2)|\bx^\top|^2.
\end{align*}
Now, we observe that \eqref{eq:defn-nc} gives
\[
|\bW| \leq (1+\mu^2)\fc(f), \qquad |\bX| \leq (1+\lambda^2)\fc(f), 
\]
so
\begin{align*}
e^{2\psi} \tilde \cQ_{\lambda,\mu}(\tilde e_1,\tilde e_2,\tilde e_3,\tilde e_4) & \geq (1+\lambda^2)(1+\mu^2)(8 - 3\fc(f)^2)\\
& \qquad - 4(1+\lambda^2)(1+\mu^2) |\bx^\perp|^2 - 16(1+\lambda^2)(1+\mu^2)|\bx^\top|^2.
\end{align*}
We may rearrange this and use $|\bx|^2 = |\bx^\top|^2 + |\bx^\perp|^2$ to write
\begin{align*}
\frac{e^{2\psi} }{(1+\lambda^2)(1+\mu^2)} \tilde \cQ_{\lambda,\mu}(\tilde e_1,\tilde e_2,\tilde e_3,\tilde e_4) & \geq 8 - 3\fc(f)^2 + 12 |\bx^\perp|^2 - 16|\bx|^2.
\end{align*}
Now assume that $\fc(f) < \sqrt\frac 43$. Using Lemma \ref{lemm:angle} (as in Proposition \ref{prop:secPetrunin}), we find $12 |\bx^\perp|^2 - 16|\bx|^2 \geq - 4$. Using this, we have $\tilde \cQ_{\lambda,\mu}(\tilde e_1,\tilde e_2,\tilde e_3,\tilde e_4) > 0$. This completes the proof.
\end{proof}

\section{Remarks and Conjectures} \label{sec:conj} 

A natural question is to understand the effects of other intrinsic curvature conditions on the minimal normal curvature of a closed manifold immersed in $\BB^N$. In addition to $\sec>0$ and strictly PIC-2, we also studied the sharp lower bound for $\fc(f)$, assuming the non-existence of metrics with positive Ricci curvature or PIC. One may speculate that under these stronger conditions might lead to better bounds as compared to Proposition \ref{prop:secPetrunin} and Proposition \ref{prop:PIC2-petrunin}. However, similar computations yield the following.

\begin{prop}
	Let $X^n$ be closed  and admits an immersion $f: X \to \BB^N$. 
	\begin{enumerate}
		\item Suppose $X$ does not admit any metric with $\Ric>0$. Then $\fc(f)<\sqrt{\tfrac 32}$.
		\item Suppose $X$ does not admit any metric with strictly PIC. Then $\fc(f)<\sqrt{\tfrac 43}$.
	\end{enumerate}
	The constants in both statements are sharp.
\end{prop}

In other words, one cannot distinguish between $\sec>0$ and $\Ric>0$, nor strictly PIC and strictly PIC-2, with only $\cC_N(X)$.

This also raises several questions (see also Remark \ref{rema:RP3}). For example, one may ask for the values of $\cC_N(S^2\times S^2)$ or $\cC_N(S^2\times T^2)$. We had initially hoped to compute a lower bound matching Proposition \ref{prop:Sn-T2} for  $\cC_N(S^2\times T^2)$ by considering bi-Ricci curvature (cf.\ \cite{ShenYe})
\[
\BiRic(e_1,e_2) = \Ric(e_1,e_1) + \Ric(e_2,e_2) - \sec(e_1,e_2)
\] 
in place of sectional/PIC-2 in Propositions \ref{prop:secPetrunin} and \ref{prop:PIC2-petrunin} (and then appeal to \cite{BHJ} for the fact that $S^2\times T^2$ does not admit positive bi-Ricci curvature). However, it seems that this does not yield the sharp bound:
\begin{conj}\label{Conj:BiRic}
If $f: X^4\to \BB^N$ has $\fc(f) < \sqrt{\frac{12}{7}}$ then $g=f^*g_{\RR^N}$ is conformally equivalent to a metric of positive bi-Ricci curvature. 
\end{conj}
Note that $\sqrt{\frac{12}{7}} \approx 1.31$ while Proposition \ref{prop:Sn-T2} gives $\cC_N(S^2\times T^2) \leq \sqrt{\frac 95} \approx 1.34$. Algebraic manipulations of the Gauss equation \eqref{eq:Gauss} show that Conjecture \ref{Conj:BiRic} holds if we assume that the immersion lies on the unit sphere $S^{N-1}$. However, it seems difficult to extend the approach used for Proposition \ref{prop:secPetrunin} and \ref{prop:PIC2-petrunin} to the general case of a map to $\BB^N$. 

We also propose the following conjecture that could theoretically yield a matching lower bound:
\begin{conj}\label{conj:imo-ric}
For $X^n$ closed, if $f: X^n\to \BB^N$ has $\fc(f) < \sqrt{\frac 95}$ then $g=f^*g_{\RR^N}$ is conformally equivalent to a metric $\tilde g$ so that if $\tilde\lambda_1 \leq \dots \leq \tilde\lambda_4$ are the eigenvalues of $\widetilde\Ric$ at any point, then $\tilde \lambda_4 < \tilde\lambda_1+\tilde\lambda_2 + \tilde\lambda_3$. 
\end{conj}
It is an interesting problem to determine whether or not $S^2\times T^2$ admits such a metric.  Note that since $\tilde\lambda_3 \leq \tilde\lambda_4$, this condition (which could be rewritten as $2\,  \widetilde{\Ric}(u) < \widetilde{\scal}$ for all $\tilde g$-unit vectors) is stronger than $2$-convexity of $\widetilde\Ric$, i.e.\ $\tilde\lambda_1 + \tilde\lambda_2 > 0$. As a first observation, by considering the totally geodesic $T^3 \subset S^2 \times T^2$, the method from \cite{Strake} implies that one cannot find a $C^1$-path of metrics $\{g_t\}_{t\in (-\eps, \eps)}$, such that $g_0$ is a product metric and $\frac{d}{dt}|_{t=0} (\scal_{g_t} - 2\Ric_{g_t})>0$ everywhere on $S^2\times T^2$.

\bibliography{bib}
\bibliographystyle{amsalpha}

\end{document}